\documentclass[letterpaper, 10 pt, conference]{ieeeconf}

\IEEEoverridecommandlockouts

\usepackage{relsize}
\usepackage{amsmath,amssymb,amsfonts,mathtools}
\usepackage{nicematrix,nicefrac}
\usepackage{algorithmic}
\usepackage{textcomp}
\usepackage{xcolor}
\usepackage{booktabs}
\usepackage{graphicx}
\usepackage{epsfig}
\usepackage{times}
\usepackage{siunitx}
\NewCommandCopy\OldSI\SI
\RenewDocumentCommand\SI{ O{} m m }{%
\nobreak%
\OldSI[#1]{#2}{#3}%
\nobreak%
}

\usepackage{tikz}
\usepackage{xstring}
\usepackage{forloop}
\usepackage{colortbl}
\usepackage{soul}
\usepackage{array}

\let\labelindent\relax
\usepackage[inline]{enumitem}

\usepackage{multirow}
\usepackage{multicol}
\usepackage{arydshln}
\usepackage{pgfgantt}
\usepackage{rotating}
\usepackage{float}
\usepackage{algorithm}
\usepackage[edges]{forest}

\usepackage[caption=false,font=footnotesize]{subfig}
\usepackage{cite} % IEEE numeric citations

\makeatletter
\let\NAT@parse\undefined
\makeatother

\usepackage[
colorlinks=true,
linkcolor=blue,
citecolor=blue,
urlcolor=blue,
pdfborder={0 0 0}
]{hyperref}

\newcommand{\g}{\textnormal{\textsl{g}}}

\renewcommand{\c}{\mathrm{c}}
\renewcommand{\d}{\nabla}

\newcommand{\s}{{\mathrm{s}}}

\renewcommand{\c}{{\mathrm{c}}}
\renewcommand{\t}{{\mathrm{t}}}
\newcommand{\norm}[1]{\left\|#1\right\|}
\newcommand{\dotp}[2]{\langle #1,#2\rangle}

\newcommand{\cC}{{\mathcal C}}

\newcommand{\cN}{{\mathcal N}}

\newcommand{\cS}{{\mathcal S}}

\newcommand{\cX}{{\mathcal X}}

\newcommand{\cZ}{{\mathcal Z}}

\newcommand{\SO}{\mathsf{SO}}

\newcommand{\eqdef}{\mathrel{\mathop:}=}

\newcommand{\R}{\mathbb R}
\newcommand{\rS}{\mathbb S}
\newcommand{\col}[1]{\mathrm{col}\!\left(#1\right)}
\newcommand{\cma}{\mathord{,}}

\DeclareMathOperator{\diag}{diag}

\makeatletter
  \renewcommand*\env@matrix[1][*\c@MaxMatrixCols c]{%
	\hskip -\arraycolsep
	\let\@ifnextchar\new@ifnextchar
	\array{#1}}
\newcommand*\bigcdot{\mathpalette\bigcdot@{1.5}}
\newcommand*\bigcdot@[2]{\mathbin{\vcenter{\hbox{\scalebox{#2}{$\m@th#1\cdot$}}}}}
\makeatother

\newcounter{excounter}[section]

\newcounter{rmcounter}

\newcounter{dfcounter}

\newcounter{thmcounter}

\newcounter{corcounter}[thmcounter]

\newcounter{prcounter}

\newcounter{clcounter}

\usepackage{amsthm}
\usepackage{mathrsfs}
\usepackage{placeins}

\newtheorem{theorem}{Theorem}
\newtheorem{lemma}[theorem]{Lemma}

\theoremstyle{definition}
\newtheorem{definition}{Definition}
\newtheorem{assumption}{Assumption}
\newtheorem{problem}{Problem}

\theoremstyle{definition}
\newtheorem{remark}{Remark}
\makeatletter

\newcommand{\Rmnum}[1]{\expandafter\@slowromancap\romannumeral #1@}
\makeatother

\usepackage[caption=false,font=footnotesize]{subfig}

\title{\LARGE \bf
Path-following Control of a Quadrotor using Quasi-Static Transverse Feedback Linearization
}
\author{Mohamed Al Lawati$^{1}$ and Adeel Akhtar$^{2}$% <-this % stops a space
\thanks{$^{1}$ Department of Mechanical and Industrial Engineering, Sultan Qaboos University, Muscat, Oman. {\tt\small mlawati@squ.edu.om}}
\thanks{$^{2}$ Department of Mechanical and Industrial Engineering, New Jersey Institute of Technology, NJ, USA. {\tt\small adeel.akhtar@njit.edu}}
}

\begin{document}

\maketitle
\thispagestyle{empty}
\pagestyle{empty}

%%%%%%%%%%%%%%%%%%%%%%%%%%%%%%%%%%%%%%%%%%%%%%%%%%%%%%%%%%%%%%%%%%%%%%%%%%%%%%%%
\begin{abstract} 

We propose a quasi-static transverse feedback linearization (QSTFL) controller for a quadrotor to follow a prescribed geometric path, rather than a time-parameterized trajectory. In contrast to existing dynamic-feedback approaches, the controller does not introduce additional controller states. The thrust input is computed algebraically from the current state, eliminating the need for thrust-derivative measurements and numerical integration. The proposed design renders the path-following manifold invariant, ensuring that trajectories initialized on the path remain on it for all future time, while simultaneously regulating tangential velocity and yaw. We establish a diffeomorphic coordinate transformation and prove local exponential stability of the path-following manifold. In addition, closed-form expressions are derived for the thrust and torque inputs. Compared with dynamic-feedback constructions, the controller requires inversion of only a $3\times 3$ decoupling matrix rather than a $4\times 4$ one, leading to a simpler control law and reduced computational complexity. Numerical simulations demonstrate the effectiveness of the proposed method. Code and animations are publicly available at \footnotesize{\texttt{\href{https://gitlab.com/a5akhtar/quasistatic-tfl-uav/}{https://gitlab.com/a5akhtar/quasistatic-tfl-uav/}}}.

\end{abstract}

%%%%%%%%%%%%%%%%%%%%%%%%%%%%%%%%%%%%%%%%%%%%%%%%%%%%%%%%%%%%%%%%%%%%%%%%%%%%%%%%
\section{Introduction}
{
For a vehicle moving along a prescribed curve, a time-parameterized reference can be unsuitable: if the vehicle slows down or stops, the reference continues to evolve and a tracking error is generated. Path following avoids this issue by removing explicit time dependence from the output reference and, with suitable feedback, yields motion invariance along the path~\cite{AkhWasNie2012}. This framework is well suited to mobile robots and, in particular, to the unmanned aerial vehicles (UAVs). In this paper, the UAV under consideration is a quadrotor. Accordingly, throughout the paper, the terms \emph{UAV} and \emph{quadrotor} are used interchangeably; when the meaning is clear from context, we also refer to it simply as the \emph{system}. A quadrotor is a nonlinear underactuated system with six degrees of freedom and four control inputs~\cite{MahKumCor2012}. 

Quadrotor motion control is commonly developed either in cascade form, with an inner attitude loop and an outer position loop~\cite{RozMag2014,BarRamBer2019}, or in monolithic form~\cite{RozMag2012SE3,MahAliAkh2025}. In cascade designs, stability of the subsystems does not by itself imply stability of the full closed loop. Monolithic designs avoid this separation and are well suited to path-following problems with motion invariance. For this reason, we adopt a monolithic control design.

A further distinction is between coordinate-based and coordinate-free quadrotor controllers~\cite{LeeLeoMcC2010a,ChaMcC2006}. Coordinate-free formulations avoid singularities associated with local attitude parameterizations, but the existing designs are typically not formulated as motion-invariant path-following controllers for the full quadrotor dynamics. Here we address path following directly for the full translational and rotational dynamics and obtain a local exponential stability result.

Our approach is based on transverse feedback linearization (TFL)~\cite{RozMag2012SE3,NieConMagTos2008}, which treats path following as a set-stabilization problem. Since quadrotors are differentially flat~\cite{AkhWasNie2012}, one can construct a feedback law and a coordinate transformation that put the dynamics into Brunovský normal form. Existing path-following designs for quadrotors use dynamic feedback~\cite{AkhWasNie2012,RozMag2012SE3}, which introduces additional controller states. In contrast, we use quasi-static feedback~\cite{DelRud98}, so the control input is expressed in terms of auxiliary inputs and their time derivatives, computed from the current state, without dynamic extension. This reduces controller complexity and simplifies implementation.

In this paper, we introduce quasi-static transverse feedback linearization (QSTFL) for quadrotor path following. Unlike existing dynamic-feedback constructions~\cite{AkhWasNie2012,RozMag2012SE3}, the proposed design does not augment the system with extra controller states. Related quasi-static feedback results for quadrotors in~\cite{FriDeBulLoh2012} do not address the full nonlinear dynamics and do not establish motion invariance along a path. In contrast, our controller is designed for the full quadrotor dynamics, follows a broad class of curves, and locally exponentially stabilizes the set of motions along the path, thereby rendering that set controlled invariant. Moreover, the controller requires inversion of a $3\times 3$ decoupling matrix rather than a $4\times 4$ one.

To the best of our knowledge, this is the first QSTFL-based controller for a quadrotor that exponentially stabilizes the motion-invariant path-following set. The main contributions are:
\begin{enumerate*}[label=(\roman*)]
    \item exponential stability of motions along the path (Theorem~\ref{thm:stab}),
    \item a diffeomorphic coordinate transformation without dynamic extension (Lemma~\ref{lem:diffeo}), and
    \item closed-form expressions for the control inputs \eqref{eq:ut} and~\eqref{eq:tau}.
\end{enumerate*}
}

\section{Notation and Math Preliminaries}
\subsection{Notation}
The set of reals is denoted by $\R$. An element $x\in\R^n$ is written in its components as $\col{x_1,x_2,\cdots,x_n}$, where $x_i$ is the $i^{\rm th}$ component of $x$. The first and second time-derivatives  are respectively denoted $\dot x$ and $\ddot x$. For $i\geq3$, the $i^\text{th}$ time-derivative  of $x$ is $x^{(i)}$. For a matrix $R$, its $i^{\rm th}$ column is $R_i$, and its $(i,j)^{\rm th}$ entry is $R_{ij}$. The transpose of $x$ is  $x^\top$. Given a set $\cS$, the restriction of $x$ to $\cS$ is denoted as $x\vert_{\cS}$. The neighborhood of $\cS$ is denoted as $\cN_{\cS}$. 
The Euclidean norm of $x$ is $\|x\|$, 
For $x, y \in \mathbb{R}^n$, the standard inner product is denoted by $\langle x, y \rangle = x^\top y$. For a continuously differentiable function $f:\mathbb{R}^n \to \mathbb{R}$, the gradient at $x \in \mathbb{R}^n$ is
\(
\nabla f(x) = (\frac{\partial f(x)}{\partial x})^\top \in \mathbb{R}^n,
\)
where $\frac{\partial f}{\partial x}$ denotes the vector of partial derivatives with respect to $x$. The Jacobian of $f$ is ${\rm d} f$. 
The unit $n$-sphere is defined as $\mathbb{S}^{n} = \{ x\in\R^{n+1} : \| x\| = 1 \}$. For vectors $x, y \in \mathbb{R}^3$, the cross product operator is defined as
\(
\times : \mathbb{R}^3 \times \mathbb{R}^3 \to \mathbb{R}^3, \;
(x, y) \mapsto x \times y
\). The natural basis in $\R^3$ is $\{e_1, e_2, e_3\}$. 
Trigonometric functions are abbreviated as $\sin(\xi)\!=\! \s_\xi$, $\cos(\xi)\!=\!\c_\xi$, and $\tan(\xi) \!=\! \t_\xi$.

\subsection{Math Preliminaries}
A \emph{smooth closed curve} in $\mathbb{R}^3$ is the image of a smooth embedding
\(
\gamma:\mathbb{S}^1 \to \mathbb{R}^3.
\)
Equivalently, it is a compact connected smooth $1$-dimensional embedded submanifold of $\mathbb{R}^3$. In particular, a smooth closed curve is compact and admits a well-defined smooth tangent direction after choosing an orientation. For example, for any $c\in\mathbb{R}$, the map
\(
\gamma_1(\lambda)=(\cos\lambda,\sin\lambda,c), \; \lambda\in\mathbb{S}^1,
\)
parametrizes the horizontal unit circle at height $c$, which is a smooth closed curve in $\mathbb{R}^3$.

%%%%%%%%%%%%
Our stability results rely on the generalized inverse function theorem~\cite{GuiPol74}, which we restate below.
\begin{theorem}[The generalized inverse function theorem~\cite{GuiPol74}] \label{thm:gift}
Let $\sigma$ be smooth map on $\cX$, and let $\cX^\ast \subset \cX$ be a closed embedded submanifold. Suppose
\begin{enumerate}[label=(\roman*)]
    \item $\sigma\vert_{\cX^\ast}: \cX^\ast \to \sigma(\cX^\ast)$ is a diffeomorphism, and
    \item for all $x^\ast \in \cX^\ast$, ${\rm d}\sigma(x^\ast)$ is invertible.
\end{enumerate}
Then, $\sigma\vert_{\cN_{\cX^\ast}}$ is a diffeomorphism.
\end{theorem}%
\noindent This theorem generalizes the inverse function theorem in the sense that  $\cX^\ast$ is a set---not a single point.

%%%%%%%%%%%%%%%%%%%%%%%%%%%%%%%%%%%%%%%%%%%%%%%%%%%%%%%%%%%%%%%%%%%%%%%%%%%%%%%%
\section{Modeling} \label{sec:modl}
%MA1: revision done. 
%\adeel{AA: As you mentioned, this paragraph needs a bit of improvement. I have tried to revise the firs 2-3 sentences.}
Consider a standard quadrotor model~\cite{MahKumCor2012,AkhWasNie2012}, and let $\mathscr{N} = \{n_1,n_2,n_3\}$ be a fixed navigation frame with origin $O_{\mathscr{N}}$, where $n_1,n_2,n_3\in\R^3$ are three orthonormal basis vectors aligned with the north, east, and down directions, respectively. Also, let $b_1,b_2,b_3\in\R^3$ form a right-handed orthonormal basis, and let $\mathscr{B}= \{b_1, b_2, b_3\}$ be a frame whose origin, $O_{\mathscr{B}}$, is fixed at the UAV's center of mass (CoM). The position and velocity, respectively, of $O_{\mathscr{B}}$ relative to $O_{\mathscr{N}}$, are denoted  $p \in \R^3$ and $v = \dot{p} \in \R^3$. Both $p$ and $v$ are expressed in $\mathscr{N}$. The attitude of $\mathscr{B}$ relative to $\mathscr{N}$ is expressed by a rotation matrix $R \in \SO(3)$, whose columns are $b_1$, $b_2$, and $b_3$. The angular velocity of $\mathscr{B}$ relative to $\mathscr{N}$ is denoted by $\Omega$, and it is expressed in the body frame $\mathscr{B}$. The UAV is actuated by positive thrust $u_t \in (0,\infty)$ and torque $\tau \in \R^3$. The thrust acts along the $-b_3$ direction, and the $i^{\rm th}$ torque component $\tau_i$ acts about $b_i$, $i=1,2,3$. The UAV mass and moment of inertia are, respectively, denoted by $m$ and $J$.
The equations of motion of a standard UAV are given by~\cite{RozMag2012SE3,MahAliAkh2025}
\begin{subequations}
	\label{eq:UAV_geo}
	\begin{align}
		\dot{p} &= v, &
		\dot{v} &= \g e_3 -  \frac{1}{m}u_t R e_3 \label{eq:vdt}  \\
		\dot{R} &= RS({\Omega}), &
		\dot{\Omega} &= J^{-1}(\tau - \Omega \times J\Omega). \label{eq:OMdt}
	\end{align}
\end{subequations} 
Notice that the coupling between the translational  dynamics, \eqref{eq:vdt} and the rotational dynamics, \eqref{eq:OMdt} occurs through $Re_3 = b_3$. Hence, among the twelve possible Euler parameterizations~\cite{SchJun18}, we choose the 2-1-3 Euler rotation sequence since it yields the simplest $b_3$, i.e., no yaw dependence, and does not contain a singularity at hover state.

Let $\theta$, $\phi$, and $\psi$ be in $\R$ and represent respectively, pitch, roll, and yaw of the UAV. We define the UAV's attitude by $\eta := \col{\theta, \phi, \psi}$. 
%Throughout the paper, we denote the $i^{\rm th}$ column of $R$ as $R_i$ using the 2-1-3 Euler parameterization, i.e., $R_i = b_i$. 
Thus, \eqref{eq:UAV_geo} is parameterized as 
%MA1: addressed
% \adeel{AA: either we should stick with Eq. (.) or just (.) while referring to an equation. I think the second one is better.}
\begin{equation}
	\label{eq:UAV_Euler}
	\begin{aligned}
		\dot{p} &= v, &
		\dot{v} &= \g e_3 - \frac{1}{m}u_t R_3, \\
		\dot{\eta} &= W\Omega, &
		\dot{\Omega} &=   J^{-1}(\tau - \Omega \times J\Omega),
	\end{aligned}
\end{equation}
where
\begin{equation*}
	 R_3 = b_3 = \left[\begin{array}{ccc}\s_\theta\c_\phi\\ -\s_\phi\\ \c_\theta\c_\phi \end{array}\right],\text{and} \;\;
	W = \left[\begin{array}{ccc}
		\s_\psi / \c_\phi & \c_\psi / \c_\phi & 0 
		\\
		\c_\psi & -\s_\psi & 0 
		\\
		\s_\psi \t_\phi & \c_\psi \t_\phi & 1 
	\end{array}\right].
\end{equation*}
Next, we define the state and input vectors, respectively, as $x \eqdef \col{p,v,\eta,\Omega} \in \cX \subset \R^{12}$ and $u \eqdef \col{u_t, \tau} \in (0,\infty) \times \R^3$. Hence, \eqref{eq:UAV_Euler} can be written in a control-affine form as
\begin{equation}
	\label{eq:UAV}
	\dot{x} = f(x) + G(x)u,
\end{equation}
where the vector fields $f:\R^{12} \to \R^{12}$ and $G:\R^{12} \to \R^{12\times4}$ are given by \small
\begin{equation*}
	 f(x) = 
	 \begin{bmatrix}
	 	v\\\g e_3 \\ W\Omega\\ -J^{-1}(\Omega \times J\Omega)
	 \end{bmatrix}, \text{ and} \quad
	  G(x) = 
	  \begin{bmatrix}
	  	0_{3\times 4}\\
	  	\begin{matrix}
	  		-Re_3/m & 0_{3\times 3}
	  	\end{matrix}\\
	  	0_{3\times 4}\\
	  	\begin{matrix}
	  		0_{3\times 1} & J^{-1}
	  	\end{matrix}
	  \end{bmatrix}.
\end{equation*}
\normalsize
The degree of underactuation of~\eqref{eq:UAV} is $2$. Also, the model~\eqref{eq:UAV} is not well-defined on the singularity point $\phi = \pm \pi/2$.
%%%%%%%%%%%%%%%%%%%%%%%%%%%%%%%%%%%%%%%%%%%%%%%%%%%%%%%%%%%%%%%%%%%%%%%%%%%%%%%%
\section{Problem formulation}\label{sec:prblm}
Given a UAV, the main objective is to drive the UAV CoM to a desired geometric path $\cC$ in 3D space while imposing invariance properties on $\cC$. By ``geometric", we mean that $\cC$ is a curve in $\R^3$ \textit{not} parameterized by time. The invariance properties simply mean that if the UAV is initialized on the path with velocity tangent to the path, it stays on the path  for all time. This is in contrast to the conventional trajectory-tracking objective where such invariance is not guaranteed.

In this work, the desired path $\cC$ is assumed to be a compact connected
smooth embedded curve in $\R^3$. Moreover, there exists an open neighborhood
$U_{\cC}\subset\R^3$ of $\cC$ and smooth functions
$h_1,h_2:U_{\cC}\to\R$ such that
\begin{equation}\label{eq:C}
\cC=\{p\in U_{\cC}: h_1(p)=h_2(p)=0\},
\end{equation}
and
\begin{equation}\label{eq:regularity_C}
\nabla h_1(p)\times \nabla h_2(p)\neq 0,
\qquad \forall p\in U_{\cC}.
\end{equation}
Hence, $\cC$ is a regular intersection of two surfaces and therefore a smooth
embedded one-dimensional submanifold of $\R^3$. After choosing an orientation,
the unit tangent field associated with $\cC$ is well defined and smooth on
$U_{\cC}$, and is given by
\begin{equation}\label{eq:rho}
\rho(p)=
\frac{\nabla h_1(p)\times \nabla h_2(p)}
{\|\nabla h_1(p)\times \nabla h_2(p)\|}.
\end{equation}
Therefore, while on the path, the tangential speed of the UAV CoM is
\(
v\vert_{\cC}=\langle v,\rho(p)\rangle .
\)
Since \eqref{eq:regularity_C} holds on $U_{\cC}$, the field $\rho$ is well defined
throughout a neighborhood of $\cC$.

\begin{assumption}\label{asp:pp}
Given the UAV model~\eqref{eq:UAV} and a path $\cC$ satisfying \eqref{eq:C}--\eqref{eq:regularity_C},
assume that the following conditions hold on an open neighborhood $\cN_{\cC}\subset U_{\cC}$
of $\cC$:
    \begin{enumerate}[label=\textbf{A\arabic*} ]
        \item \label{asp:ut}$u_t \neq 0$ and
        \item \label{asp:dt} $\dotp{\d h_2}{R_3} \neq 0$.
    \end{enumerate}
\end{assumption}

Assumption~\ref{asp:ut} implies that there is always nonzero thrust during motion. Indeed, $u_t=0$ corresponds to all motors being shut off, which is undesirable for safe flight operations.  As for~\ref{asp:dt}, on $\cC$, we realize that $\dotp{\d h_2}{R_3} = 0$ corresponds to $R_3$ being in a tangent plane to $h_2=0$, meaning the thrust produces no force in the direction needed to influence $h_2$, causing a singularity in~\eqref{eq:ut}.

Let us equip~\eqref{eq:UAV} with an output $h \in \R^4$, defined as 
\begin{equation}
h = \col{
	h_1(p), h_2(p), \dotp{v}{\rho} - v_{d},
	\psi - \psi_{d}
},
\end{equation}
where $v_{d}, \psi_{d} : \R^3 \to \R$ represent, respectively, the desired tangential velocity of UAV CoM and desired UAV yaw at a given point $p$ on the path $\cC$. Note that when $h_1=h_2=0$, the third output component $\dotp{\rho}{v} - v_d$ measures the deviation of the UAV velocity tangent to the path from the desired tangential velocity $v_d$. With the above setup, the path-following problem (PFP) can be formulated as follows.
\begin{problem}[Path Following Problem (PFP)]
\label{prob:Quasi-PF-Problem} 
Given the UAV model expressed in control-affine form~\eqref{eq:UAV} and a path~\eqref{eq:C} expressed as a set of points in the output space without a notion of time, design a feedback controller %$\kappa: \R^{12} \to \R^4$, $x\mapsto \kappa(x) = u$,
such that the closed-loop UAV system achieves the following goals:
\begin{enumerate}[label=\textbf{G\arabic*} ]
	\item \label{goal:1} The UAV CoM approaches $\cC$ exponentially fast.
	\item \label{goal:2} The UAV CoM's motion along $\cC$ is controlled invariant.
	\item \label{goal:3} The controller %$\kappa(x)$
    achieves two additional control specifications while motion is on $\cC$, namely, $\dotp{v}{\rho} \to v_{d}$ and $\psi \to \psi_{d}$ as $t \to \infty$. 
\end{enumerate}
\end{problem}
\begin{remark}\label{rmk:dea}
The authors~\cite{RozMag2012, AkhWasNie2012} solved a similar PFP but rely on dynamic extension, introducing thrust and its time-derivative as controller states.  This augmentation has two drawbacks: recovering thrust requires either solving a differential equation or integrating twice from sensor measurements, which poses practical challenges. Also, gain tuning must account for two extra states. The solution presented in this paper overcomes these challenges.
\end{remark}

%%%%%%%%%%%%%%%%%%%%%%%%%%%%%%%%%%%%%%%%%%%%%%%%%%%%%%%%
%%%%%%% Solution using QSTFL  
%%%%%%%%%%%%%%%%%%%%%%%%%%%%%%%%%%%%%%%%%%%%%%%%%%%%%%
\section{Solution using a QSTFL}\label{sec:fdk}
Intuitively, our solution to the PFP is attained by driving $h$ to zero using a feedback. When $h=0$, the feedback should maintain this value for all time. We rely on QSF to feedback linearization output dynamics~\cite{DelRud98}, from which we adapt the following definition
\begin{definition}[Quasi-static feedback]
Given a dynamical system $\dot{x}=f(x,u)$, a feedback is called quasi-static if it is of the form
\[
    u_i = \varphi_i(x,\nu,\dot{\nu}, \dots, \nu^{(s)}), \quad i = 1,\dots, m
\]
where $s$ is a finite integer,  $\varphi_i$ are smooth functions, and
$\nu$ is an auxiliary input such that  
\[
 \nu = \tilde{\varphi}(x).
\]
\end{definition}

Taking time-derivatives of each output until an input appears yields the decoupling matrix $A$, whose invertibility guarantees input-output linearization. Specifically, differentiating $h_1$, $h_2$, $h_4$ twice and $h_3$ once gives
\begin{equation*}
    A = \left[\!\!\!\!\begin{array}{c:c}
  -\dfrac{D^\top \! R_3}{m} & 
  \begin{matrix} 0_{3\times3} \\[6pt] e^\top W  J^{-1}\end{matrix}
\end{array}\!\!\!\!\right] \!,  
D = \left[\d h_1, \d h_2, \rho, -\d \psi_d\right].
\end{equation*}
Notice that $A$ is singular of constant rank 2 since, away from the model singularity $\phi=\pm \pi/2$, two rows are always independent. 
This means that we have two  affine independent relations between the inputs and  output derivatives. Furthermore, we realize that at hover, i.e., UAV is horizontally oriented at rest, the decoupling matrix $A$ becomes
\begin{equation*}
    A_0  = 
    \begin{bmatrix}
       -\frac{\partial h_1 / \partial p_3}{m}&0&0&0\\[5pt]
       -\frac{\partial h_2 / \partial p_3}{m}&0&0&0\\[5pt]
      0&0&0&0\\[5pt]
       0&0&0&\frac{1}{J_{33}}
    \end{bmatrix}.
\end{equation*}

We note that the independence of the surfaces $h_1$ and $h_2$ implies that  both $\partial h_1 / \partial p_3$ and $\partial h_2 / \partial p_3$ are not zero simultaneously. We choose $h_2$ such that $\partial h_2 / \partial p_3 \neq 0$ at hover; 
the same construction applies if $h_1$ is chosen instead. Hence, the expressions of 
%either $\ddot{h}_1$ or 
$\ddot{h}_2$ along with $\ddot{h}_4$ form two independent relations with the input $u$. These two relations are valid at the hover case. 
Therefore, we assign two auxiliary inputs $\nu_2$ and $\nu_4$ as
\vspace{-10pt}

    \begin{align}\label{eq:nu24}
        \nu_2 &= \ddot{h}_2,   &
         \nu_4  &= \ddot{h}_4. 
    \end{align}

If $\nu_2$ and $\nu_4$ were known, then~\eqref{eq:nu24} provides two affine equations in terms of the input $u$. To design $\nu_2$ and $\nu_4$, let's define the error coordinates $z^i = \col{h_i, \dot{h}_i}$, for $i\in \{2,4\}$. Hence~\eqref{eq:nu24} results in 
\begin{equation}\label{eq:zi24}
    \dot{z}^i = A_{c2}z^i + B_{c2}\nu_i,
\end{equation}
where $(A_{c2}$, $B_{c2})$ is a 2-by-2 Brunovsky pair.

Next, we show that one can readily compute the thrust $u_t$ using~\eqref{eq:nu24} and the choice of $\nu_2$ made in~\eqref{eq:nu}. 
Let $H$ be the Hessian of $h_2$ with respect to $p$. Thus, using the system dynamics and~\eqref{eq:nu24}, we have
$
    \nu_2 = \ddot{h}_2 = v^\top H v + \dotp{\d h_2}{e_3}\g - \dotp{\d h_2}{R_3}u_t/m,
$
which can be solved for the thrust $u_t$ giving
\begin{equation}\label{eq:ut}
    u_t = u_t(x, \nu_2) = m\frac{v^\top H v  - \nu_2 + \g (\partial h_2 / \partial p_3)}{\dotp{\d h_2}{R_3}}.
\end{equation}

\begin{remark}
    In this design, we compute the thrust algebraically by performing elementary algebraic operations. This is unlike the dynamic controller. See Remark~\ref{rmk:dea}. The condition $u_t > 0$ in \eqref{eq:ut} is only guaranteed near $\mathcal{X}^*$. For large initial errors, $K_2$ 
    should be designed to avoid $u_t \leq 0$, or $\nu_2$ can be 
    saturated to enforce $u_t > 0$ at the cost of temporarily 
    losing the exponential stability guarantee.
\end{remark}
Similarly, letting $H_{d}$ denote the Hessian of $\psi_{d}$ with respect to $p$, we realize that 
\begin{equation}
\label{eq:nu4expr}
    \begin{aligned}
        \nu_4 = \ddot{h}_4 &= e_3^\top  \left( W J^{-1}\tau + \dot{W}\Omega - W J^{-1}(\Omega\times J\Omega)\right) \\
        &- v^\top H_{d} v + \dotp{R_3}{\d \psi_{d}} u_t/m - \g \dotp{\d \psi_{d}}{e_3}.
    \end{aligned}
\end{equation}
This yields only one equation in $\tau$. We differentiate further to obtain $h_1^{(4)}$ and $h_3^{(3)}$, giving two additional independent equations in $\tau$. Thus, we define auxiliary inputs $\nu_1$ and $\nu_3$

\begin{equation*}
        \nu_1 = h_1^{(4)},\quad
        \nu_3 = h_3^{(3)}.
\end{equation*}

The error coordinates are defined as $z^1 = \mathrm{col} \bigl(h_1\cma \dot{h}_1\cma \ddot{h}_1\cma h_1^{(3)}\bigr)$ and $z^3 = \col{h_3\cma \dot{h}_3\cma \ddot{h}_3}$. As a result, the corresponding error dynamics becomes
\begin{equation}\label{eq:z1z3}
    \begin{aligned}
        \dot{z}^1 &= A_{c4}z^1 + B_{c4}\nu_1,&
        \dot{z}^3 &= A_{c3}z^3 + B_{c3}\nu_3,
    \end{aligned}
\end{equation}
where $(A_{c4}, B_{c4})$ and $(A_{c3}, B_{c3})$ are, respectively, the 4-by-4 and 3-by-3 Brunovsky pairs. 
\begin{remark}
    Note that the four decoupled dynamics in~\eqref{eq:zi24} and \eqref{eq:z1z3} are linear-time invariant. Thus, there exists a plethora of inputs $\nu_i$ for each decoupled system to stabilize its error dynamics.
\end{remark}

We now form the other two equations. Computing two additional time-derivatives for $\ddot{h}_1$ and $\dot{h}_3$, and realizing that $u_t = u_t(x,\nu_2)$, we get
\begin{align*}
    h_1^{(4)} &= h_1^{(4)}(x,\nu_2,\dot{\nu}_2,\ddot{\nu}_2), &
    h_3^{(3)} &= h_3^{(3)}(x,\nu_2,\dot{\nu}_2).
\end{align*}
Direct computation reveals
\begin{subequations}\label{eq:nu13}
    \begin{align}
        \nu_1 = h_1^{(4)} &=  
        \frac{u_t(\d h_1 \times \d h_2)^\top}{m\dotp{\d h_2}{R_3}}\begin{bmatrix}
            \frac{ R_1}{J_{11}}   & \frac{ R_2}{J_{22}} &0\end{bmatrix}\tau \nonumber\\  &+\tilde{b}_1(x{,\nu_2,\dot{\nu}_2, \ddot{\nu}_2})\\
            \nu_3 = h_3^{(3)} &= \frac{ u_t \rho^\top}{ m\dotp{\d h_2}{R_3}}\begin{bmatrix}
            \frac{(\d h_2 \times R_1)}{J_{11}} & \frac{(\d h_2 \times R_2)}{J_{22}} & 0\end{bmatrix}\tau  \nonumber\\
            & +\tilde{b}_3(x,\nu_2,\dot{\nu}_2),
    \end{align}
\end{subequations}
where the expressions of $\tilde{b}_1$ and $\tilde{b}_3$ are lengthy and, hence, omitted, and $u_t$ is precomputed in~\eqref{eq:ut}. As a result, combining~\eqref{eq:nu4expr} and \eqref{eq:nu13}, we arrive at
\begin{equation}
\label{eq:feedback_3by3}
    \begin{bmatrix}
        \nu_1\\ \nu_3 \\ \nu_4
    \end{bmatrix} = \tilde{A} \tau + \begin{bmatrix}
        \tilde{b}_1\\
        \tilde{b}_3\\
          \dotp{\d \psi_{d}} {  R_3 \frac{u_t}{m}- \g  e_3}- v^\top H_{d} v
        \end{bmatrix},
\end{equation}
where $\tilde{A}$ is the new decoupling matrix given by
\begin{equation}\label{eq:Atilde}
    \begin{aligned}
        \tilde{A} =
        \begin{bmatrix}
           u_t\frac{\dotp{\d h_1 \times \d h_2}{R_1}}{mJ_{11}\dotp{\d h_2}{R_3}}   & u_t\frac{\dotp{\d h_1 \times \d h_2}{R_2}}{mJ_{22}\dotp{\d h_2}{R_3}} &0 \\[8pt]
           u_t\frac{\dotp{\rho}{\d h_2 \times R_1}}{mJ_{11}\dotp{\d h_2}{R_3}} & u_t\frac{\dotp{\rho}{\d h_2 \times R_2}}{mJ_{22}\dotp{\d h_2}{R_3}} & 0 \\[8pt]
           \frac{\s_{\psi}\t_{\phi}}{J_{11}} & \frac{\c_{\psi}\t_{\phi}}{J_{22}} & \frac{1}{J_{33} }
        \end{bmatrix}.
    \end{aligned}
\end{equation}
\begin{remark}
    This decoupling matrix is 3-by-3 as compared to the 4-by-4 decoupling matrix of~\cite{RozMag2012} and~\cite{AkhWasNie2012}. This dimension simplification is a result of precomputing the thrust in~\eqref{eq:ut}.
\end{remark}

It can be seen from~\eqref{eq:feedback_3by3} that to calculate the input $\tau\in\R^3$, we need to take the inverse of the decoupling matrix $\tilde{A}$. It suffices to investigate the determinant of $\tilde{A}$ given by 
\begin{equation}
    \begin{aligned}
        \det( \tilde{A} ) &=
        - \frac{\norm{\d h_1 \times \d h_2}u_t^2}{m^2J_{11}J_{22}J_{33}\dotp{\d h_2}{R_3}}.
    \end{aligned}
\end{equation}
Assumption~\ref{asp:pp} and the smooth embedded-curve assumption on $\cC$ ensure that $\det(\tilde A) \neq0$ in a neighborhood of the curve. Thus, the torque can be written in terms of state $x$ and auxiliary input $\nu$ as
\begin{align}\label{eq:tau}
 \tau &= \tau(x,\nu,\dot{\nu}_2,\ddot{\nu}_2) \\
 &= \tilde{A}^{-1}\left(
 \begin{bmatrix}
        \nu_1\\ \nu_3 \\ \nu_4
    \end{bmatrix}
    - \begin{bmatrix}
        \tilde{b}_1(x,\dot{\nu}_2,\ddot{\nu}_2)\\
        \tilde{b}_3(x,\dot{\nu}_2)\\
          \dotp{\d \psi_{d}} {  R_3 \frac{u_t}{m}- \g  e_3}- v^\top H_{d} v
        \end{bmatrix}
        \right),\nonumber
\end{align}
where $u_t$ is computed in~\eqref{eq:ut} and the auxiliary inputs $\nu_i$ are given in~\eqref{eq:nu}. 
The auxiliary input choice is made as
\begin{equation}\label{eq:nu}
    \nu_i = -K_iz^i, \quad i=1,2,3,4,
\end{equation}
where $K_1 \in \R^{1\times 4}$, $K_3 \in \R^{1\times 3}$, and $K_2,K_4 \in \R^{1\times 2}$ are stabilizing gain row vectors chosen such that the linear error dynamics~\eqref{eq:zi24} and \eqref{eq:z1z3} are exponentially stable with desired eigenvalues. Notice that the error dynamics, given by~\eqref{eq:zi24} and \eqref{eq:z1z3}, is 11-dimensional. Hence, there exists $1$-dimensional zero dynamics whose stability is to be investigated.

\section{Stability analysis}
\label{sec:stability_analysis}
The 11-dimensional error dynamics~\eqref{eq:zi24},~\eqref{eq:z1z3} are exponentially stabilized by~\eqref{eq:nu}, leaving a 1-dimensional zero-dynamics~\cite{Isi95} whose stability we now analyze. Let $z(x) = \col{z^1(x),z^2(x),z^3(x),z^4(x)} \in \R^{11}$. The zero-dynamics evolve on the zero-dynamics manifold, $\cX^\ast$, given by
\begin{equation}
\label{eq:X-star}
    \cX^\ast = \{x \in \cX: z(x) = 0\}.
\end{equation}

Let the state $x\in \cX \subset \R^{12}$ be transformed to  $(z, \xi) \in \cZ \subset \R^{12}$, i.e., $x\mapsto (z,\xi)$ under some coordinate transformation, where the scalar $\xi$ is to be defined. 

Since the $z$-coordinate captures the error dynamics, which is transverse to $\cC$, a plausible choice for the remaining coordinate, $\xi$, is a one that encodes the motion tangent to $\cC$.
To define $\xi$, we recall that $\cC$ is a smooth embedded-curve by definition. Hence, $\cC$ is diffeomorphic to the unit circle $\rS^1$, and let $\Lambda: \cC \to \rS^1$ be the diffeomorphism. Now, define the projection map $\pi : \cX \to \cC$, $x \mapsto \arg \min_{\tilde{p}\in \cC} \left(\norm{p - \tilde{p}}\right)$, which takes a point in the state-space and returns the closest point on the path $\cC$. With this setup, we define $\xi$ as
\begin{equation}
     \xi = \lambda (x),
\end{equation}
where $\lambda = \Lambda \circ \pi$.
Thus, the new state $  \xi$ represents the arc-length on $\rS^1$ measured from a fixed point on $\rS^1$ in the counterclockwise direction. Since $\cC$ is a smooth compact embedded curve in $\R^3$, by the tubular neighborhood theorem~\cite{GuiPol74}, there exists $\varepsilon > 0$ such that $\pi$ is a smooth well-defined map on $\cN_{\cX^\ast} := \{x \in \cX : \norm{p}_{\cC} < \varepsilon\}$, and we restrict all subsequent results to this neighborhood. Furthermore, for any $x^* \in \cX^*$ with position component $p \in \mathcal{C}$, we have $\pi(x^*) = p$ since $p$ is 
its own closest point on $\mathcal{C}$, and hence 
$\pi|_{\cX^*}$ is onto $\mathcal{C}$. We now define the candidate coordinate transformation $\sigma$ as
\begin{equation}
\label{eq:sigma}
\begin{aligned}
    \sigma: \cN_{\cX^\star} \subset \cX &\to \sigma(\cN_{\cX^\star}) \subset \cZ 
          &x \mapsto \col{z,   \xi} .
    \end{aligned}
\end{equation}

Note that $(z,  \xi) \in \sigma(\cN_{\cX^\star}) \subset \cZ$. For $\sigma$ to be a coordinate transformation, it must be a diffeomorphism between 
$\cN_{\cX^\ast}$
and 
$\sigma{(\cN_{\cX^\ast})}$.
This diffeomorphism will be established using Theorem~\ref{thm:gift}, which requires $\cX^\ast$ to be compact -- a fact established in the following lemma.

\begin{lemma}\label{lem:diffeoactness}
    The zero-dynamics set $\cX^\ast$ is compact.
\end{lemma}
\begin{proof}
     It suffices to show that $\cX^\ast$ is diffeomorphic to a compact set. This set can be chosen to be the unit circle $\rS^1$. The diffeomorphism will be constructed using the restriction of $\sigma$ on $\cX^\ast$, denoted as $\sigma\vert_{\cX^\ast}$. From the definition of $\sigma$, we have $\sigma\vert_{\cX^\ast} = \col{0,\dots,0,\lambda\vert_{\cX^\ast}}$. 
     %MA1: You are absolutely right :)
     %\adeel{We use $\eta$ to denote three Euler angles earlier. Should it be $\lambda$?}. 
     The mapping $\lambda\vert_{\cX^\ast}$ is composed of two smooth and onto maps\footnote{In fact, the mapping $\Lambda$ is already a diffeomorphism, but $\pi$ is not. Particularly, $\pi$ is not one-to-one. For example, take $\tilde{p}\in \cC$, which can be mapped from $\col{p,v_1,\eta_1,\Omega_1}$ and $\col{p,v_2,\eta_2,\Omega_2}$. In fact, infinitely many such points exist. However, $\pi$ is onto since it covers the entire path $\cC$.}, hence, it is smooth and onto. Thus, we only need to show that $\pi\vert_{\cX^\ast}$ is one-to-one. Now, suppose $p \in \cC$. We need to show that there exists a unique $x^\ast \in \cX^\ast$ such that $\pi(x^\ast)= p$. To do so, we will express $x^\ast$ as a function of $p$. On $\cX^\ast$, we have $h_3 = 0$. Thus, $\dotp{\rho(p)}{v} = v_{d}(p)$. Also, $h_3=0$ implies that velocity of the UAV CoM is tangent to the path $\cC$. This allows us to write $v = \norm{v}\rho$. Since $\norm{\rho}=1$, we immediately have $\norm{v}=v_d(p)$. Thus, $ \| v(p) \| \rho(p) = v_{d}(p) \rho(p)$, which gives
     \begin{equation}\label{eq:vrho}
         v(p) = v_{d}(p) \rho(p).
     \end{equation}
     Hence, $v$ is a function of $p$. From~\eqref{eq:UAV_Euler}, we have $\dot{v} = \g e_3 - u_t/mR_3$ implying \vspace{-5pt}
     \begin{equation}\label{eq:tmp}
        m(\g e_3 - \dot{v}) = u_t R_3.
    \end{equation}
     Since $\norm{R_3}=1$, we arrive at $u_t = \vert u_t \vert = \norm{m(\g e_3 - \dot{v})}$. From~\eqref{eq:vrho}, we have $\dot{v}=v_{d}(p) \dot{\rho}(p) + \dot{v}_{d}(p) \rho(p)$ implying $u_t =\norm{m(\g e_3 - v_{d}(p) \dot{\rho}(p) - \dot{v}_{d}(p)\rho(p))}$. From~\eqref{eq:tmp}, we get 
     \begin{align*}
         \phi \!&= \!\sin^{-1}\!\!\left(\!m\frac{\dot{v}_2}{u_t}   \!\right) \!\! =\! \sin^{-1}\!\!\left(\frac{\left(v_{d}(p) \dot{\rho}(p) + \dot{v}_{d}(p) \rho(p)\right)\cdot e_2}{\norm{\g e_3 - v_{d}(p) \dot{\rho}(p) - \dot{v}_{d}(p)\rho(p)}}   \right)\\
         \theta &= \tan^{-1}\left(  \frac{\left(v_{d}(p) \dot{\rho}(p) + \dot{v}_{d}(p) \rho(p)\right)e_1}{\left(v_{d}(p) \dot{\rho}(p) + \dot{v}_{d}(p) \rho(p)\right)e_3 - \g} \right).
     \end{align*}     
     Also on $\cX^\ast$, $h_4=0$, and therefore, $\psi = \psi_{d}(p)$. Finally, invoking $\Omega = W^{-1}\dot{\eta}$ and using the expressions of $\phi$, $\theta$, and $\psi$ just found, we can write $\Omega = \Omega(p)$. As a result, $x^\ast$ is uniquely determined by $p$. Hence, the mapping $\pi\vert_{\cX^\ast}$ is one-to-one. This result implies that $\sigma\vert_{\cX^\ast}$ is a diffeomorphism, and hence, $\cX^\ast$ and $\rS^1$ are diffeomorphic. As a result, compactness of $\cX^\ast$ immediately follows from the compactness of  $\rS^1$.
\end{proof}

Lemma~\ref{lem:diffeoactness} guarantees the compactness of $\cX^\ast$. We need this property of $\cX^\ast$ to apply Theorem~\ref{thm:gift}. The next lemma uses Theorem~\ref{thm:gift} to prove that $\sigma$ is a diffeomorphism from $\cN_{\cX^\ast}$ to $\sigma(\cN_{\cX^\ast})$.

%In the next lemma, we prove that $\sigma$ is a diffeomorphism using the generalized inverse function theorem~\cite{GuiPol74}.
\begin{lemma}\label{lem:diffeo}
    The map $\sigma: \cN_{\cX^\star} \subset \cX \to \sigma(\cN_{\cX^\star}), x \mapsto (z,  \xi)$ is a diffeomorphism of a neighborhood of $\cX^\ast$ onto its image. 
\end{lemma}
\begin{proof}
    Lemma~\ref{lem:diffeoactness} showed that $\cX^\ast$ is compact. As well, its proof showed that $\sigma\vert_{\cX^\ast}$ is a diffeomorphism, and hence one-to-one. We now invoke Theorem~\ref{thm:gift}. This means that to show $\sigma$ is a diffeomorphism, we need to show that 
    $\det(\rm d \sigma)$ is not zero on $\cX^\ast$. %The determinant of the Jacobian of $\rm d \sigma$ can be computed as
    We compute
    \begin{equation*}
         \det  (\mathrm{d} \sigma)  = \frac{u_t^4 \norm{\d h_1 \times \d h_2}^3 \dotp{\d h_1 \times \d h_2} {\d \lambda} \c_{\phi}}{m^4 \langle R_3, \d h_2\rangle ^2}.
    \end{equation*}
    Based on the tubular neighborhood established 
    above, $\cC$ being a smooth embedded-curve, and by Assumption~\ref{asp:pp}, we immediately have $u_t \neq 0$, $\norm{\d h_1 \times \d h_2} \neq 0$, and $\langle R_3, \d h_2\rangle \neq 0$. In addition, $\phi=\pm \pi/2$ is a model singularity, and hence $\c_{\phi}\neq0$ in the region of model validity.     
    It is left to show that $\dotp{\d h_1 \times \d h_2} {{\d } \lambda} \neq 0$. 
    We have ${\d}\lambda\vert_{\cX^\ast} = {\d}\Lambda \circ 
    {\rm d}\pi\vert_{\cX^\ast} = {\d}\Lambda$. Since $\Lambda$ is a diffeomorphism, ${\d }\Lambda$ 
    is a nonzero tangent to $\cC$. We conclude that ${\d}\lambda$, $\d h_1$, and $\d h_2$ 
    are independent vectors. Hence, $\dotp{\d h_1 \times \d h_2} {{\d} \lambda} 
    \neq 0$.
\end{proof}
Lemma~\ref{lem:diffeo} shows that $\sigma$ is a diffeomorphism onto its image, hence, it is a valid coordinate transformation. We use this coordinate transformation to prove that when the transformed state $z \to 0$ this implies that the system state $x \to \cC$. This fact is formalized in our main result, which is stated in the next theorem.
\begin{theorem}\label{thm:stab}
For the closed-loop system $\dot z = A_cz+B_c\nu$, where
$A_c = \operatorname{blkdiag}(A_{c4},A_{c2},A_{c3},A_{c2})$ and $\nu$ given by~\eqref{eq:nu}, the set $\cX^\ast$ given in~\eqref{eq:X-star} is exponentially stable. Moreover, the transformed state $\lambda$ is bounded.
\end{theorem}
\begin{proof}
    We write $\cX^\ast$ in the $(z, \xi)$-coordinates as
    $
        \sigma(\cX^\ast) = \{(z, \xi) \in \cZ : z = 0\}.
    $
    The $z$-subsystem is made exponentially stable by the choice of inputs~\eqref{eq:nu}. The $\xi$-subsystem 
    evolves on a compact set $\rS^1$. Hence, the $\lambda$-dynamics is bounded. Therefore, $\sigma(\cX^\ast)$ is exponentially stable. Since $\sigma$ is a diffeomorphism by Lemma~\ref{lem:diffeo}, $\cX^\ast$ is exponentially stable.
\end{proof}

Theorem~\ref{thm:stab} proves that $\cX^\ast$ is exponentially stable. The exponential stability of $z^1$-dynamics and $z^2$-dynamics mean that any initial condition close enough to $\cC$ results in a motion that converges to $\cC$. Moreover, any motion starts in $\cX^\ast$ remains in $\cX^\ast$. As a result objective \ref{goal:1} and \ref{goal:2} are fulfilled. The exponential stability of $z^3$-dynamics and $z^4$-dynamics guarantees that $v \to v_d \rho$ and $\psi \to \psi_d$. This fulfills objective~\ref{goal:3}. Hence, the control input given by thrust~\eqref{eq:ut} and torque~\eqref{eq:tau}, where $\nu$ is designed according to~\eqref{eq:nu}, indeed solves the PFP.

%%%%%%%%%%%%%%%%%%%%%%%%%%%%%%%%%%%%%%%%%%%%%%%%%%%%%%%%%%%%%%%%%%%%%%%%%%%%%%%%
\section{Simulation}
\label{sec:simulations}
This section demonstrates the QSF controller on two example paths: a horizontal circle and a twisted curve. %Our controller ensures convergence to $\cC$ and invariance of $\cX^\ast$. 
Using QSTFL, motion invariance is highlighted in subsection~\ref{sec:eg1} and path convergence is illustrated in subsection~\ref{sec:eg2}. To have more realistic simulations in both examples, the QSF is 
tested on the non-parameterized UAV model given in~\eqref{eq:UAV_geo}. The UAV mass and inertial matrix are chosen as $m = \SI{2}{\kilogram}$ and $J = \diag{(0.09,0.09,0.15)}\unit{\kilogram \cdot \meter\squared}$. The proposed controller is simulated in MATLAB\footnote{Code and animation at \hyperlink{https://gitlab.com/a5akhtar/quasistatic-tfl-uav/}{https://gitlab.com/a5akhtar/quasistatic-tfl-uav/}.}.

\subsection{Example 1: horizontal circle}\label{sec:eg1}
The path is a horizontal circle with constant traversal velocity, where the UAV heading is directed toward the center at all times — a configuration relevant to applications such as firefighting.   Fig.~\ref{fig:inv_screenshot} illustrates the invariance of $\cX^\ast$. The UAV is initialized on $\cC$ at rest with the heading pointing outward, i.e., away from the center. The heading angle is corrected, while the UAV never leaves the path $\cC$. Fig.~ \ref{fig:inv_inputs} shows the corresponding input that generated this motion.

\begin{figure}[ht]
    \centering
    \subfloat[Snapshots of motion.\label{fig:inv_screenshot}]{%
        \includegraphics[width=0.5\columnwidth]{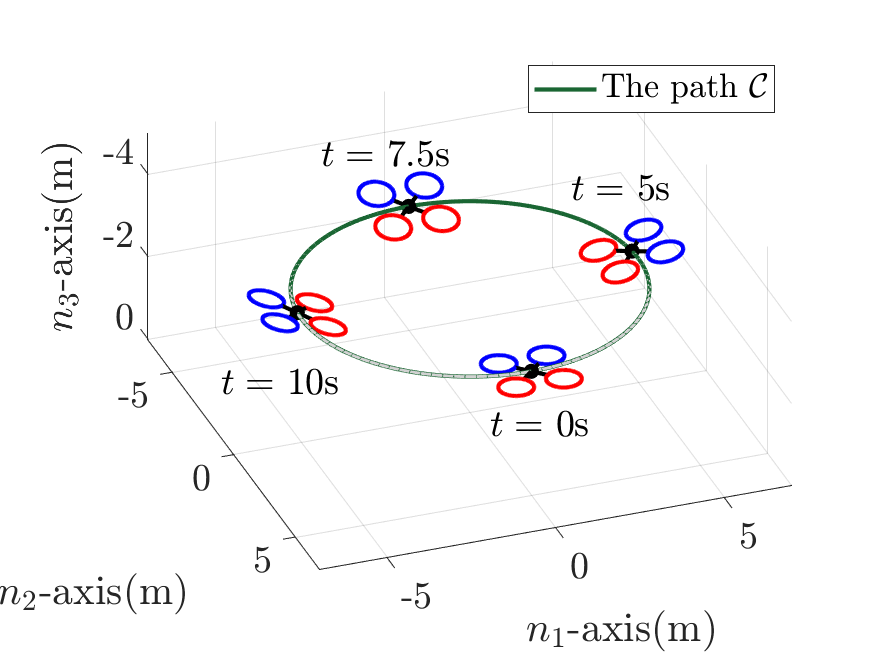}}%
    \hfill%
    \subfloat[Input signal.\label{fig:inv_inputs}]{%
        \includegraphics[width=0.5\columnwidth]{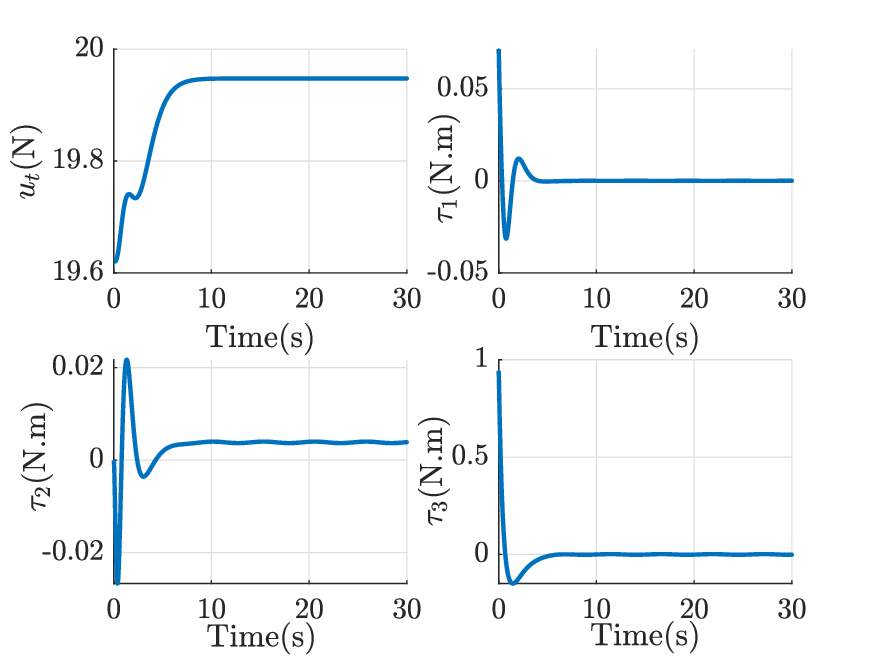}}%
    \caption{Controlled-invariance results under the QSTFL controller.}
    \label{fig:inv_combined}
\end{figure}

\subsection{Example 2: twisted path}\label{sec:eg2}
Consider a path formed by the intersection of the following two surfaces $h_1 = p_1 ^ 2 + p_2 ^ 2 - 4$ and $h_2 = -\sin(2p_1) + p_2 + p_3 + 3$. On $\cC$, the desired velocity and heading angle are chosen, respectively,  $v_d = \SI{1}{\meter\per\second}$ and $\psi_d = \SI{0}{\radian}$. %Hence, the output function is written as
Hence, the components of the output function $h$ are $h_1 =  p_1 ^ 2 + p_2 ^ 2 - 4$, $h_2 = -\sin(2p_1) + p_2 + p_3 + 3$, $h_3 \!\!\!\!\!\!\!\!\!\! = \!\!\!\!\!\!\!\!\!\!\!\! \frac{2 \cos  \left(2 p_{1} \right) p_{2} v_{3} +\left(-v_{2} +v_{3} \right) p_{1} +p_{2} v_{1}}{\sqrt{4 \left(\cos^{2}\left(2 p_{1} \right)\right) p_{2}^{2}+4 \cos  \left(2 p_{1} \right) p_{1} p_{2} +2 p_{1}^{2}+p_{2}^{2}}}-1$, and $h_4 = \psi$.
The UAV is initialized as $p(0) = \col{4 \cma 5 \cma -1}\unit{\meter}$, $v(0) = \col{-0.1 \cma 0.2 \cma -0.2} \unit{\meter\per\second}$, $\eta(0) = \col{-\pi/10 \cma \pi/8 \cma -\pi/4} \unit{\radian}$, and $\Omega(0) = \col{-0.9 \cma 0.1 \cma 0.5}\unit{\radian\per\second}$.
Table~\ref{tab:gains} shows the gains along with eigenvalues of all the closed-loop $z^i$-dynamics.
\begin{table}[H]
    \centering
    \caption{Gains and eigenvalues of the $z^i$-dynamics \eqref{eq:z1z3} and \eqref{eq:zi24}.}
    \label{tab:gains}
    \begin{tabular}{ccc}
        \toprule
        {Dynamics} &           Gains, $K_i$ & Eigenvalues \\
        \midrule
        $z^1$ & $[6.24\cma 16.76\cma 16.22 \cma    6.7]$ & $\{-1\cma -2\cma -1.3\cma -2.4\}$\\ 
        $z^2$ & $[2\cma 3]$ & $\{-1\cma -2\}$\\ 
        $z^3$ & $[2.6\cma   5.9 \cma    4.3]$ & $\{-1\cma -2\cma -1.3\}$\\ 
        $z^4$ & $[2\cma 3]$ & $\{-1\cma -2\}$\\
        \bottomrule
    \end{tabular}
\end{table}
With the above data, the motion in Fig.~\ref{fig:snapshot_conv} is generated, where the path convergence is illustrated. 
The output signals are shown in Fig.~\ref{fig:output_conv}.
\begin{figure}[!htbp]
    \centering
    \subfloat[The QSTFL guarantees convergence to path.]{%
        \includegraphics[width=0.5\columnwidth]{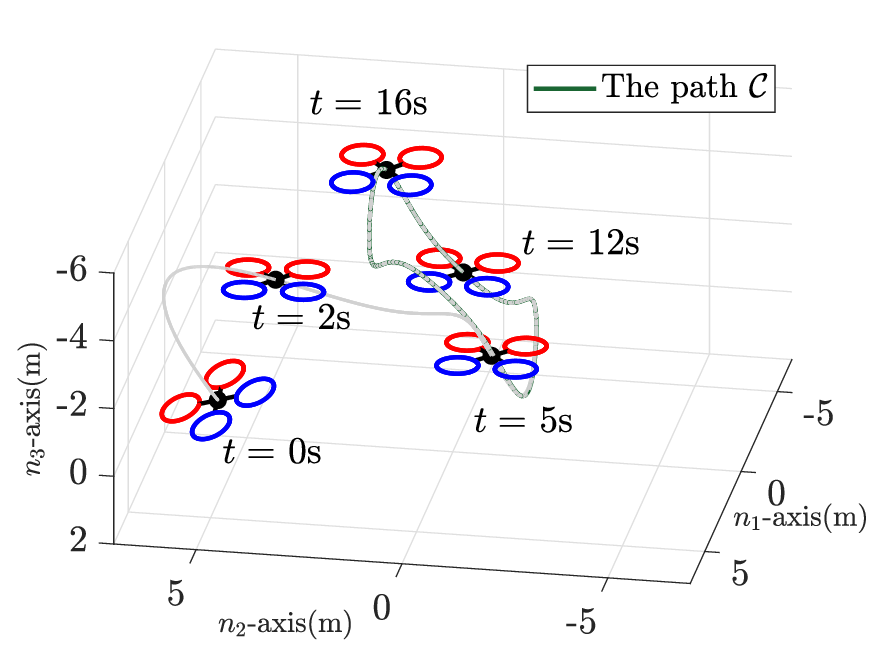}%D
         \label{fig:snapshot_conv}}
    \hfill
    \subfloat[The output signal.]{%
        \includegraphics[width=0.5\columnwidth]{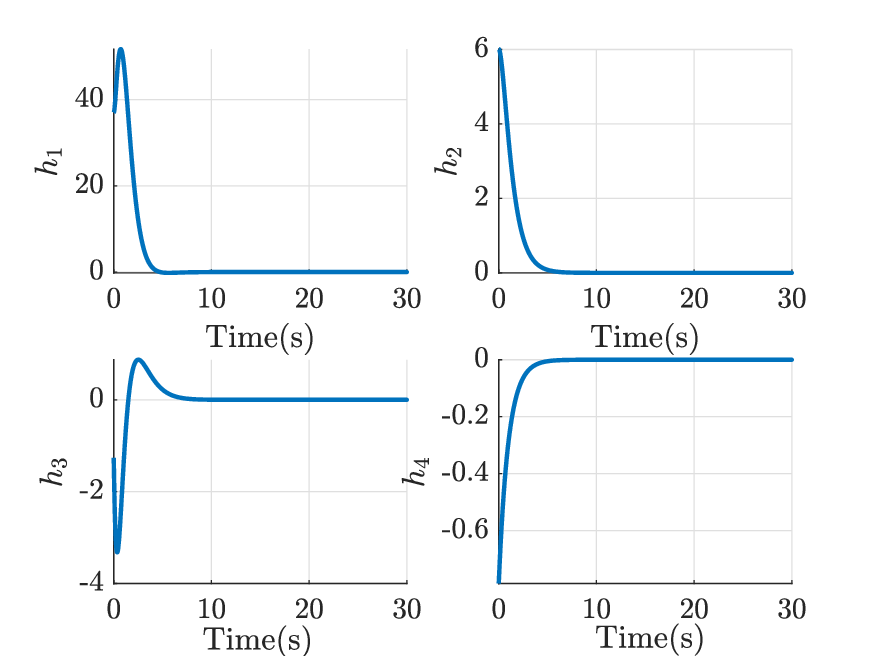}%
        \label{fig:output_conv}}
    \caption{Snapshots, and input/output time signals for example 2.}
    \label{fig:two_column_subfigs}
\end{figure}

%%%%%%%%%%%%%%%%%%%%%%%%%%%%%%%%%%%%%%%%%%%%%%%%%%%%%%%%%%%%%%%%%%%%%%%%%%%%%%%%

\section{Conclusions}
\label{sec:conclusion}

This paper presented a QSTFL controller for the quadrotor path following problem, achieving exponential stability, controlled-invariance, and direct tangential velocity and yaw control along the path. Unlike dynamic feedback designs \cite{AkhWasNie2012, RozMag2012}, the proposed controller does not require additional states, eliminating thrust-derivative measurements, double integration, and simplifying gain tuning. The thrust force is computed algebraically, reducing the decoupling matrix from $4\times 4$ to $3\times 3$ and lowering computational complexity. Theoretical guarantees include a diffeomorphic coordinate transformation, compactness of the zero-dynamics manifold, and local exponential stability of the path-following set (Lemma~\ref{lem:diffeoactness}, Lemma~\ref{lem:diffeo}, and Theorem~\ref{thm:stab}). Numerical simulations validated the theoretical results, demonstrating both motion invariance and path convergence. In summary, this paper presents a local solution to Problem~\ref{prob:Quasi-PF-Problem} using quasi-static feedback.

\addtolength{\textheight}{-2cm}   % This command serves to balance the column lengths

%%%%%%%%%%%%%%%%%%%%%%%%%%%%%%%%%%%%%%%%%%%%%%%%%%%%%%%%%%%%%%%%%%%%%%%%%%%%%%%%

%%%%%%%%%%%%%%%%%%%%%%%%%%%%%%%%%%%%%%%%%%%%%%%%%%%%%%%%%%%%%%%%%%%%%%%%%%%%%%%%

% \printbibliography

{\footnotesize
\bibliographystyle{IEEEtran}
\bibliography{conf-abbr,journal-abbr,ref,adeelbib}
}

%%%%%%%%%%%%%%%%%%%%%%%%%%%%%%%%%%%%%%%%%%%%%%%%%%%%%%%%%%%%%%%%%%%%%%%%%%%%%%%%
%\section*{APPENDIX}

%\section*{ACKNOWLEDGMENT}

\end{document}